\documentclass[english,a4paper]{article}
\usepackage[T1]{fontenc} 
\usepackage{textcomp,lmodern} 
\usepackage[utf8]{inputenc} 
\usepackage{babel} 
\usepackage{mathtools} 
\usepackage{amssymb} 
\usepackage{amsthm} 
\usepackage{xcolor} 
\usepackage{tikz}
\usetikzlibrary{graphs,graphs.standard,positioning}
\tikzset{mynode/.style={shape= circle, fill = white, inner sep = 0pt, outer sep = 0.3pt, minimum size = 4pt,draw,thick}}
\tikzset{myloop below/.style={thick,loop, out=-50, in = -130,min distance =6mm}}
\newtheorem{lemma}{Lemma}
\newtheorem{corollary}[lemma]{Corollary}
\newtheorem{proposition}[lemma]{Proposition}

\theoremstyle{definition}
\newtheorem{definition}[lemma]{Definition}
\newtheorem{remark}[lemma]{Remark}
\DeclareMathOperator\Aut{Aut}
\DeclareMathOperator\cayl{Cayl}
\DeclareMathOperator\inv{inv}
\DeclareMathOperator\schrei{Schrei}
\DeclareMathOperator\Stab{Stab}
\DeclareMathOperator\Star{Star}
\newcommand*{\Cayley}[2]{\cayl(#1;#2)}
\newcommand*{\Schreier}[3]{\schrei(#1,#2;#3)}
\newcommand*{\SchreierOr}[3]{\schrei_{\mathcal O}(#1,#2;#3)}
\newcommand*\Z{\mathbf{Z}}

\usepackage{hyperref}  
\title{Up to a double cover, every regular connected graph is isomorphic to a Schreier graph}
\author{Paul-Henry Leemann\thanks{The~author was supported by NSF Grant No. $200021\_188578$.}}
\date{\today}
\hypersetup{pdftitle={Up to a double cover, every regular connected graph is isomorphic to a Schreier graph}, pdfauthor={Paul-Henry Leemann}, pdfsubject={Schreier graphs, Cayley graphs, regular graphs, perfect matchings, coverings}}

\begin{document}
\maketitle
\begin{abstract}
We prove that every connected locally finite regular graph is either isomorphic to a Schreier graph, or has a double cover which is isomorphic to a Schreier graph.
\end{abstract}
\textit{AMS classification number:}  05C25
\section{Introduction}
A central tool in combinatorial and geometric group theory is the notion of the Cayley graph of a group $\Gamma$, which is a graph that encodes the structure of the group, see Section~\ref{Sec:1} for all the relevant definitions.
This construction is enlightening in two ways: on one hand it allows to see groups as geometric objects, while on the other hand it endows some graphs with an algebraic structure.
In both direction, it permits a better understanding of the objects under consideration.
It is therefore natural to ask which graphs can be realized as a Cayley graph of some group.
This question has been answered in 1958 by Sabidussi~\cite{MR0097068} who showed that a connected graph $G$ is isomorphic to a Cayley graph if and only if $\Aut(G)$ acts freely and transitively on the vertices of $G$.

It follows from Sabidussi's result that Cayley graphs are rather sparse among graphs, and thus too restrictive to study general graphs, even if we restrict ourself to the specific case of connected regular degree graphs. A more general notion of a ``graph of algebraic origin'' is hence needed.
A suitable generalization of Cayley graphs, is the concept of a Schreier graph associated to a pair $(\Gamma,\Lambda)$ where $\Lambda$ is a subgroup of $\Gamma$ (or equivalently associated to a transitive action of $\Gamma$ on some set $X$). In this context, Cayley graphs correspond to pairs $(\Gamma,\{1\})$ (equivalently to free and transitive actions).
Schreier graphs are always connected and of regular degree and one can thus wonder which connected graphs of regular degree are isomorphic to a Schreier graph.

The first step in this direction was made by Gross in his 1977 article~\cite{MR0450121}, where he proved that every finite, connected regular graph of even degree is isomorphic to a Schreier graph.
The result for infinite connected regular graphs of even degree follows by a compacity argument and was probably one of this well-known folklore result for long.
A written proof of the infinite version can be found in~\cite{MR2893544}, but it has since be rediscovered many times.
On the other hand, it is easy to find examples, even finite ones, of connected regular graphs of odd degree that are not isomorphic to a Schreier graph, see the left of Figure~\ref{Figure:Simple3reg}.
It hence remains to classify which connected regular graphs of odd degree are isomorphic to a Schreier graph. 
This question turns out to be equivalent to the existence of a perfect matching, an interesting subject in its own right, and also to the existence of a covering map to some particular graph, see Section~\ref{Sec:2}.
An interesting result in this direction is the fact that all transitive graphs (with at most countably many vertices and edges) are isomorphic to a Schreier graph.
This fact was first proved by Godsil and Royle for finite graphs~\cite{GodsilRoyle}, then extended to locally finite graphs by the author~\cite{LeemannThese}, and finally to countable graphs by Hamann and Wendland in the appendix of~\cite{2020arXiv200706432G}.

Even if some connected regular graphs of odd degree are not isomorphic to a Schreier graph, one can still ask if they roughly ``look like'' a Schreier graph.
One possible interpretation of the above is to ask for the existence of a finite degree covering from a Schreier graph to $G$.
The aim of this note is to provide a proof of the following result, which concludes the above discussion.
\begin{proposition}\label{Prop:Schreier}
Let $G$ be a $d$- regular connected graph.
Then either $G$ is isomorphic to a Schreier graph or $G$ has a double-cover $H$ which is isomorphic to a Schreier graph.
\end{proposition}
While we were not able to find a reference to the above result in the literature, we do not claim any priority on it.
In fact, this note was inspired by the following remark, which can be found at the end of Section 7 of~\cite{MR2087782}: ``In fact up to the cover of degree $2$ any regular graph can be realized as a Schreier graph~\cite{MR1308046}''.
However,~\cite{MR1308046} seems to treat only the case of graphs of even degree.

\paragraph{}
Section~\ref{Sec:1} contains all the definitions as well as a discussion on the unusual concept of degenerated loop. The short Section~\ref{Sec:2} contains more materials on coverings and perfect matchings, as well as the proof of Proposition~\ref{Prop:Schreier}.

\paragraph{Acknowledgements}
The author is thankful to A. Georgakopoulos, R. Grigorchuk and M. de la Salle for comments on a previous version of this note.
\section{Preliminary definitions and results}\label{Sec:1}
Since one of our motivations is the study of Schreier graphs, we will look at graphs in a broad sense.
We allow multi-edges and loops as well as the less usual degenerated loops (also called half-edges, dangling edges or open edges).

Informally, a degenerated loop is the index $2$ quotient of a path of length~$1$, see Figure~\ref{Figure:DoubleCovers}.
Heuristicaly it can either be seen as an edge with only one end, or as a loop that add only one to the degree.
Allowing degenerated loops is necessary to treat the case of Schreier graphs of groups with generators of order~$2$.
\begin{definition}
A \emph{graph} is a $5$-tuple $(V,\vec E,\iota,\tau,\inv)$ where $V$ and $\vec E$ are sets (of \emph{vertices} and of \emph{arcs}) and $\iota,\tau\colon \vec E\to V$ and $\inv\colon \vec E\to\vec E$ are applications (\emph{initial vertex}, \emph{terminal vertex} and \emph{inverse arc}) such that $\iota\cdot\inv=\tau$, $\tau\cdot\inv=\iota$ and $\inv$ is an involution.
\end{definition}
There is a natural equivalence relation on arcs defined by $\vec e\sim\vec f$ if and only if $\vec e=\vec f$ or $\vec e=\inv(\vec f)$.
The quotient set $\vec E/\sim$ is denoted by $E$ and is the set of \emph{edges}.

A vertex $v$ and an edge $e=\{\vec e,\inv(\vec e)\}$ are \emph{adjacent} if $v=\iota(\vec e)$ or $v=\tau(\vec e)$, in which case we also say the $v$ is an \emph{end} of $e$.
An edge that is adjacent to only one vertex is called a \emph{loop}.
A loop $e=\{\vec e,\inv(\vec e)\}$ is \emph{degenerated} (or an \emph{half-edge}) if $\vec e=\inv(\vec e)$, that is if $\vec e$ is a fixed point of $\inv$.

\begin{remark}
Graph theorists usually forbid the existence of degenerated loops, but these special kinds of edges naturally arise in the context of Schreier graphs.
Indeed, when allowing degenerated loops, there is a bijection between the subgroups of $G$ (up to conjugacy) and graphs that are covered by $\Cayley{\Gamma}{S}$, see~\cite{MR3463202} for a proof.
\end{remark}

The \emph{degree} of a vertex is the number of arcs $\vec e$ such that $v=\iota(\vec e)$.
In particular, a non degenerated loop adds $2$ to the degree of its adjacent vertex, while a degenerated loop adds only $1$.
All graphs under consideration will be \emph{locally finite}, that is the degree of any vertex is finite.
Classical notions like \emph{paths} and \emph{connected components} are defined in a natural way.
In particular, a \emph{bipartite graph} has no loop.
Finally, for a vertex $v$ we define $\Star(v)\coloneqq\{\vec e\in \vec E\ |\ \iota(\vec e)=v\}$.

On pictures, we will represent degenerated loops  by dotted line and other edges (including non-degenerated loops) by plain lines.
\begin{figure}[htbp]
\centering
\begin{tikzpicture}[node distance=1cm,every node/.style=mynode]
\node (00){};
\node[fill=black] (01)[right of=00]{};
\node[fill=black] (10)[below of=00]{};
\node (11)[right of=10]{};
\node (v0)[below =1.5cm of 10]{};
\node[fill=black] (v1)[below =1.5cm of 11]{};
\graph[edge={thick}]{(00)--(01)--(11)--(10)--(00)};
\graph[edge={thick,bend left}]{(v0)--(v1)--(v0)};
\draw[->>,thick] (0.5,-1.3) -- (0.5,-2.1);
\begin{scope}[xshift=3.5cm]
\node (0){};
\node (1)[below of=0]{};
\node (v)[below =1.5cm of 1]{};
\graph[edge={thick, bend left}]{(0)--(1)--(0)};
\path	(v) edge[myloop below](v);
\draw[->>,thick] (0,-1.3) -- (0,-2.1);
\begin{scope}[xshift=2.5cm]
\node (0){};
\node (1)[below of=0]{};
\node (v)[below =1.5cm of 1]{};
\graph[edge={thick}]{(0)--(1)};
\path	(v) edge[myloop below, densely dotted,thick](v);
\draw[->>,thick] (0,-1.3) -- (0,-2.1);
\end{scope}
\end{scope}
\end{tikzpicture}
\caption{Double covers of some graphs; the middle and the right ones are canonical double covers.
The leftmost one is isomorphic to the double cover $\Cayley{\Z/4\Z}{\{\pm1\}}\twoheadrightarrow\Schreier{\Z/4\Z}{\Z/2\Z}{\{\pm1\}}$, the middle one to $\Schreier{\Z/4\Z}{\Z/2\Z}{\{\pm1\}}\twoheadrightarrow\Schreier{\Z/4\Z}{\Z/4\Z}{\{\pm1\}}$ and the rightmost one to $\Cayley{\Z/2\Z}{\{\pm1\}}\twoheadrightarrow\Schreier{\Z/2\Z}{\Z/2\Z}{\{\pm1\}}$.}
\label{Figure:DoubleCovers}
\end{figure}
A \emph{morphism} of graph from $(V,\vec E)$ to $(W,\vec F)$ is a map $\varphi\colon V\sqcup\vec E\to W\sqcup\vec F$ which is compatible with the graph structure.
For any vertex $v\in V$, this induces a map $\varphi_v\colon\Star(v)\to\Star\bigl(\varphi(v)\bigr)$.
A morphism $\varphi\colon G\to H$ between two graphs is a \emph{covering} if all the induced map $\varphi_v$ are bijections.
If moreover $\lvert\varphi^{-1}(v)\rvert=2$\footnote{The quantity $\lvert\varphi^{-1}(v)\rvert$ depends only on the connected component of $v$.} for every vertex, then it is called a \emph{double cover}.
It is an easy exercise to show that every $2d$-regular graph without degenerated loops covers a graph with $1$ vertex and $d$ loops.

Given a graph $G=(V,\vec E)$, its \emph{canonical double cover} is the tensor product $G\otimes K_2$ where $K_2$ has two vertices and an edge between them.
That is, the vertex set of $G\otimes K_2$ is $V\times \{0,1\}$, and an arc $\vec e$ from $v$ to $w$ is lifted to an arc from $(v,0)$ to $(w,1)$ and to an arc from $(v,1)$ to $(w,0)$.
In particular, a non-degenerated loop adjacent to $v$ is lifted to a pair of parallel edges between $(v,0)$ and $(v,1)$, while a degenerated loop adjacent to $v$ is lifted to a single edge between $(v,0)$ and $(v,1)$.
See Figure~\ref{Figure:DoubleCovers}.
A simple verification shows that $G\otimes K_2$ is a bipartite graph, which is a double cover of $G$. Moreover, $G\otimes K_2$ is connected if and only if $G$ is connected and not bipartite.

\begin{definition}
A \emph{matching} (also called a \emph{$1$-factor}) of a graph $G$ is a subgraph $M$ of $G$ such that every vertex of $M$ has degree $1$ in $M$.
A \emph{perfect matching} is a matching of $G$ that contains all vertices of $G$.
A graph is \emph{matchable} if it admits a perfect matching.
Two matchings $M_1$ and $M_2$ are \emph{orthogonal} if they share no edge.
\end{definition}
Observe that not every regular graphs of odd degree are matchable, even among graphs without degenerated loop.
See Figure~\ref{Figure:Simple3reg}.
However, every vertex-transitive graph of odd degree is matchable,~\cite{GodsilRoyle,LeemannThese}.
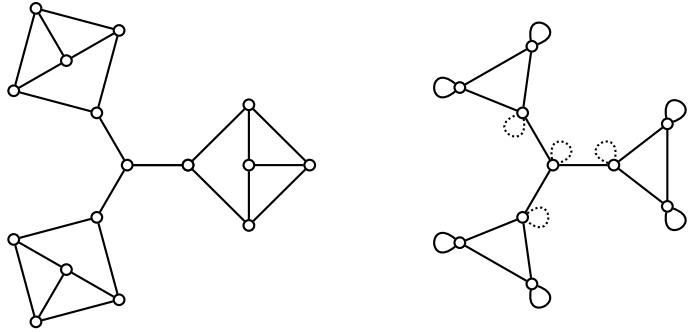
\begin{figure}[htbp]
\centering
\begin{tikzpicture}[every node/.style={mynode},scale=0.8]
\node(A) {};
\foreach \k in {0, 1, 2}{
	\begin{scope}[rotate=\k*120]
	\node[fill=white] (B\k) at (1,0){};
	\node (C\k) at (2,1){};
	\node (D\k) at (2,-1){};
	\node (E\k) at (2,0){};
	\node (F\k) at (3,0){};
	\end{scope}
	\graph[edge={thick}]{
	(A)--(B\k)--(C\k)--(F\k)--(D\k)--(B\k);
	(C\k)--(E\k)--(D\k);
	(E\k)--(F\k)
	};
}
\graph[edge={thick}]{(A)--(B0)};
\node (b0) at (1,0){};
\begin{scope}[xshift=7cm]
\node(A) {};
\foreach \k in {0, 1, 2}{
	\node (B\k) at (\k*120:1){};
	\node (C\k) at (\k*120-20:2){};
	\node (D\k) at (\k*120+20:2){};
	\graph[edge=thick]{
	(B\k)--(C\k)--(D\k)--(B\k)};
	\path (B\k) edge[loop, out=\k*120+160, in = \k*120+80,min distance =6mm, densely dotted,thick](B\k);
	\path (C\k) edge[loop, out=\k*120-20, in = \k*120-100,min distance =6mm, thick](C\k);
	\path (D\k) edge[loop, out=\k*120+100, in = \k*120+20,min distance =6mm,thick](D\k);
}
\graph[edge=thick]{
(B0)--(A)--(B1);
(B2)--(A);
};
\path (A) edge[loop, out=100, in = 20,min distance =6mm, densely dotted, thick](A);\end{scope}
\end{tikzpicture}
\caption{A simple $3$-regular graph (on the left) and a $4$-regular graph (on the right); both of them being not isomorphic to a Schreier graph.}
\label{Figure:Simple3reg}
\end{figure}

We finally recall the definition of Schreier graphs.
\begin{definition}
Let $\Gamma$ be a group with a symmetric generating set $S$ not containing the identity and let $\Lambda$ be a subgroup of $\Gamma$.
The corresponding (right) \emph{Schreier graph} $\Schreier{\Gamma}{\Lambda}{S}$ is the graph with vertex set the right cosets $\{\Lambda g\ |\ g\in\Gamma\}$ and with an arc from $\Lambda g$ to $\Lambda h$ for every $s$ in $S$ such that $\Lambda gs=\Lambda h$.
If $X$ is a set with a transitive right action $X\curvearrowleft \Gamma$, the corresponding (right) \emph{orbital graph} $\SchreierOr{\Gamma}{X}{S}$ is the graph with vertex set $X$ with an arc from $x$ to $y$ for every $s$ in $S$ such that $x.s=y$.
\end{definition}

These two definitions are the two faces of the same coin. Indeed, we have $\SchreierOr{\Gamma}{X}{S}=\Schreier{\Gamma}{\Stab_G(x)}{S}$ for every $x\in X$ and $\Schreier{\Gamma}{\Lambda}{S}=\SchreierOr{\Gamma}{\Lambda \backslash\Gamma}{S}$.
See Figure~\ref{Figure:DoubleCovers} for some examples.
Schreier graphs are generalizations of the so-called \emph{Cayley graphs}, where $\Cayley{\Gamma}{S}\coloneqq\Schreier{\Gamma}{\{1\}}{S}$.

In the graph $\Schreier{G}{H}{S}$ there is a loop at the vertex $Hg$ if and only if there exists $s$ in $S\cap g^{-1}Hg$.
This loop is degenerated if and only if $s$ is of order $2$.
If $G$ is a (finitely generated) group, then it admits a (finite) generating system without element of order $2$ if and only if it is not generalized dihedral~\cite{MR0498225}.
Moreover, even for non-generalized dihedral group, it is sometimes natural to look at generating set containing elements of order $2$.
This is the main reason why we allow degenerated loops in our definition of a graph.


\section{Coverings and matchability}\label{Sec:2}
First of all we have the following, which relates covers and perfect matchings:
\begin{lemma}\label{lemma:CoversMatch}
Let $G$ be a $(2d+1)$-regular graph without degenerated loop.
Then $G$ is matchable if and only if it covers a graph with one vertex, one degenerated loop and $d$ non-degenerated loops.
\end{lemma}
\begin{proof}
If $\varphi\colon G\to R$ is a covering and $R$ has a degenerated loop $e$, then $\varphi^{-1}(e)$ is a perfect matching of $G$.

On the other hand, let $M$ be a perfect matching of $G$ and $G'$ the graph obtained by removing all edges of $M$ from $G$.
Then $G'$ is $2d$-regular without degenerated loop and hence covers $R'$ a graph with one vertex and $d$ non-degenerated loops.
Let $R$ be the graph obtained by adding a degenerated loop to $R'$.
The covering $G'\twoheadrightarrow R'$ naturally extends to $G\twoheadrightarrow R$ by sending all edges of $M$ to the degenerated loop of $R$.
\end{proof}
As an easy generalization we obtain
\begin{corollary}
Let $G$ be a $(2d+n)$-regular graph without degenerated loop.
Then $G$ admits $n$ pairwise orthogonal perfect matchings if and only if it covers a graph with one vertex, $n$ degenerated loop and $d$ non-degenerated loops.
\end{corollary}

By a theorem of Gross~\cite{MR0450121}, every finite $2d$-regular connected graph without degenerated loop is isomorphic to a Schreier graph of the free group~$F_d$.
This result extends by compacity to infinite $2d$-regular connected graph without degenerated loop, see~\cite{MR2893544} for a proof.
On the other hand, Gross' Theorem does not carry over graphs of odd regular degree, see Figure~\ref{Figure:Simple3reg}.
However, every connected transitive graph of degree $2d+1$ and without degenerated loop is isomorphic to a Schreier graph of $F_d*(\Z/2\Z)$,~\cite{GodsilRoyle,LeemannThese}.
For general connected regular graphs, the situation is more complicated.
\begin{lemma}\label{Lemma:Isomorphism}
Let $G$ be a connected $(2d+1)$-regular graph without degenerated loop.
Then $G$ is isomorphic to a Schreier graph if and only if it is matchable.
In this case it is isomorphic to a Schreier graph of $F_d*(\Z/2\Z)$, where $F_d$ is the free group of rank $d$.
\end{lemma}
A proof of this fact can be found in~\cite{LeemannThese}, but the result itself was probably already general knowledge at that time and essentially (at least for finite graphs) follows from~\cite{MR0450121}.
A variation of Lemma~\ref{Lemma:Isomorphism} gives us
\begin{proposition}[\cite{LeemannThese}]\label{Prop:Matchings}
Let $G$ be a connected $(2d+n)$-regular graph without degenerated loop.
Then $G$ is isomorphic to a Schreier graph of $F_d*(\Z/2\Z)^{*n}$ if and only if it has $n$ perfect matchings that are pairwise orthogonal.
\end{proposition}
In order to prove Proposition~\ref{Prop:Schreier}, we will need the following fact, which follows (for example) from K\"onig's duality theorem for infinite bipartite graphs,~\cite{MR734985}.
It also follows, by compacity, from the similar statement for finite graphs.
\begin{lemma}
Every $d$-regular bipartite graph is matchable.
\end{lemma}
We now prove a more precise version of Proposition~\ref{Prop:Schreier}.
\begin{lemma}\label{Prop:Main}
Let $G$ be a $d$-regular connected graph.
If $G$ is bipartite, then it is isomorphic to a Schreier graph of $(\Z/2\Z)^{*d}$.
Otherwise, $G\otimes K_2$ is isomorphic to a Schreier graph of $(\Z/2\Z)^{*d}$.
\end{lemma}
\begin{proof}
If $G=(V,\vec E)$ is bipartite let $K\coloneqq G$, while if $G$ is not bipartite take $K$ to be its canonical double cover.

Now, the graph $K$ is bipartite and $d$-regular.
Let $M_1$ be a perfect matching of $K$ and let $K_1$ be the graph obtained from $K$ by removing the edges in $M_1$.
Then $K_1$ is bipartite of degree $d-1$ and hence has a perfect matching $M_2$ as soon as $d-1>0$.
It is trivial that $M_1$ and $M_2$ are orthogonal.
By repeating this we obtain $d$ pairwise orthogonal perfect matchings. By Proposition~\ref{Prop:Matchings} this implies that $K$ is isomorphic to a Schreier graph of $(\Z/2\Z)^{*d}$.
\end{proof}
\begin{corollary}
Let $G$ be a connected regular graph.
\begin{enumerate}
\item If $G$ has degree $2d$, then either $G$ (if it is bipartite) or $G\otimes K_2$ (otherwise) is isomorphic to a Schreier graph of $F_d$,
\item If $G$ has degree $2d+1$, then either $G$ (if it is bipartite) or $G\otimes K_2$ (otherwise) is isomorphic to a Schreier graph of $F_d*(\Z/2\Z)$.
\end{enumerate}
\end{corollary}
\begin{proof}
This follows from Lemma~\ref{Prop:Main} and the fact that a Schreier graph of $(\Z/2\Z)^{*2d}$ (respectively of $(\Z/2\Z)^{*(2d+1)}$) without degenerated loop is isomorphic to a Schreier graph of $F_d$ (respectively of $F_d*(\Z/2\Z)$),~\cite{LeemannThese}.
\end{proof}
Observe that a Schreier graph is necessarily connected and of regular degree.
On the other hand, there are examples of connected regular graphs of odd degree and without degenerated loop that are not isomorphic to a Schreier graph, see the left of Figure~\ref{Figure:Simple3reg}.
There are also examples of connected regular graphs of even degree (but with at least one degenerated loop) that are not isomorphic to a Schreier graph, see the right of Figure~\ref{Figure:Simple3reg}.
Finally, it isn't possible to ensure that $G$ has a double cover isomorphic to a Schreier graph.
Indeed, if $T$ is a regular tree, then its only double cover is the disjoint union of two copies of $T$, which is not connected and hence not isomorphic to a Schreier graph.
Hence, Proposition~\ref{Prop:Schreier} is the ``best possible'' for locally finite graphs.
%
%
%
%
%
\providecommand{\noopsort}[1]{} \def\cprime{$'$}

\textsc{Paul-Henry Leemann, Institut de Math\'ematiques, Universit\'e de Neuch\^atel, Rue Emile-Argand 11, 2000 Neuch\^atel, Switzerland}
\\\textit{Email address:} \texttt{paul-henry.leemann@unine.ch}

\end{document}